\documentclass{amsart}
\usepackage{amsmath, amssymb,epic,graphicx,mathrsfs,enumerate}
\usepackage[all]{xy}

\usepackage{amsthm}
\usepackage{amssymb}
\usepackage{latexsym}
\usepackage{longtable}
\usepackage{epsfig}
\usepackage{amsmath}
\usepackage{hhline}
\usepackage{amscd}
\usepackage{newlfont}

\usepackage{color}

 \DeclareMathOperator{\Sym}{Sym}
 \DeclareMathOperator{\soc}{soc}\DeclareMathOperator{\Gal}{Gal}

 \DeclareMathOperator{\Frat}{Frat}
 
 \DeclareMathOperator{\Stab}{Stab}

\DeclareMathOperator{\charac}{char}

\DeclareMathOperator{\core}{Core}

\DeclareMathOperator{\alt}{Alt}
\DeclareMathOperator{\Aut}{Aut}
\DeclareMathOperator{\End}{End} \DeclareMathOperator{\h}{{H^1}}
 
 \DeclareMathOperator{\diag}{Diag}
\DeclareMathOperator{\Hc}{H}

 \newcommand{\ol}{\overline}

\newtheorem{thm}{Theorem}

 \newtheorem{lemma}[thm]{Lemma}
\newtheorem{prop}[thm]{Proposition}

\numberwithin{equation}{section}

\renewcommand{\footnote}{\endnote}

\begin{document}
\bibliographystyle{amsplain}
\title[Large Chebotarev invariant]{Finite groups with large Chebotarev invariant\thanks{Research partially supported by MIUR-Italy via PRIN Group theory and applications.}}
\author{Andrea Lucchini}
\address{Universit\`a degli Studi di Padova, Dipartimento di Matematica,\\ Via Trieste 63, 35121 Padova, Italy}
\author{Gareth Tracey}
\address{Department of Mathematical Sciences, University of Bath,\\
	Bath BA2 7AY, United Kingdom}


\begin{abstract}A subset $\{g_1, \ldots , g_d\}$ of a finite group $G$ is said to invariably generate $G$ if the set $\{g_1^{x_1}, \ldots, g_d^{x_d}\}$ generates $G$ for every choice of $x_i \in G$. The Chebotarev invariant $C(G)$ of $G$ is the expected value of the random variable $n$ that is minimal
	subject to the requirement that $n$ randomly chosen elements of $G$ invariably generate $G$. The authors recently showed that for each $\epsilon>0$, there exists a constant $c_{\epsilon}$ such that $C(G)\le (1+\epsilon)\sqrt{|G|}+c_{\epsilon}$. This bound is asymptotically best possible. In this paper we prove a partial converse: namely, for each $\alpha>0$ there exists an absolute constant $\delta_{\alpha}$ such that if $G$ is a finite group and $C(G)>\alpha\sqrt{|G|}$, then $G$ has a section $X/Y$ such that $|X/Y|\geq \delta_{\alpha}\sqrt{|G|}$, and $X/Y\cong \mathbb{F}_q\rtimes H$ for some prime power $q$, with $H\le \mathbb{F}_q^{\times}$.\end{abstract}
\maketitle
\section{Introduction}
Following \cite{ig} and \cite{dixon}, we say that a subset $\left\{g_{1},g_{2},\hdots,g_{d}\right\}$ of a group $G$ \emph{invariably generates} $G$ if $\left\{g_{1}^{x_{1}},g_{2}^{x_{2}},\hdots,g_{d}^{x_{d}}\right\}$ generates $G$ for each $d$-tuple $(x_{1},x_{2}\hdots,x_{d})\in G^{d}$. The \emph{Chebotarev invariant} $C(G)$ of $G$ is the expected value of the random variable $n$ which is minimal subject to the requirement that $n$ randomly chosen elements of $G$ invariably generate $G$.

Motivated by the problem of finding field extensions $K/F$ such that a fixed finite group $G$ occurs as the Galois group of $K/F$, E. Kowalski and D. Zywina carried out a detailed investigation of the invariant $C(G)$ in \cite{kz}. Amongst many interesting results, they show that $C(G)$ can be quite large in comparison to $|G|$. More precisely, it is shown that if $G\cong G_q:= \mathbb{F}_q\rtimes\mathbb{F}^{\times}_q$, then
$$C(G)=q-\sum_{1\neq d\mid q-1}\frac{\mu(d)}{q(1-d^{-1})(1-d^{-1}+q^{-1})}.$$
In particular, $C(G_q)\sim \sqrt{|G_q|}$ as $q\rightarrow \infty$. It was also conjectured in \cite{kz} that these are the ``worst" cases: that is, that $C(G)=O(\sqrt{|G|})$ as $|G|\rightarrow\infty$. The conjecture was proved by the first author in \cite{AL}, and was later improved in \cite{ALGT} where it is shown that for each $\epsilon>0$, there exists a constant $c_{\epsilon}$ such that $C(G)\le (1+\epsilon)\sqrt{|G|}+c_{\epsilon}$. Furthermore, one has $C(G)\le \frac{5}{3}\sqrt{|G|}$ when $G$ is soluble. 

In this paper, we prove a partial converse. Informally, we prove that the only examples where $C(G)$ is a constant times $\sqrt{|G|}$ are those groups with a ``large" section isomorphic to a subgroup of $G_q$, for some prime power $q$. Our main result reads as follows.
\begin{thm}\label{MainTheorem} Fix a constant $\alpha>0$. There exists absolute constants $\beta_{\alpha}$, $\gamma_{\alpha}$, $\delta_{\alpha}$ and $k_{\alpha}$, depending only on $\alpha$, such that whenever $G$ is a finite group with the property that $C(G)>\alpha\sqrt{|G|}$, then $G$ has a factor group $\ol{G}$ such that\begin{enumerate}[(i)]
\item $\ol{G}\cong V\rtimes H$, with $V\cong \mathbb{F}_q^{k}$, and $H\le \Gamma L_1(q)\wr\Sym(k)$, with $q$ a prime power and $k\le k_{\alpha}$;
\item $|\ol{G}|\geq \delta_{\alpha}\sqrt{|G|}$; and
\item $\beta_{\alpha}|V|\le |H|\le \gamma_{\alpha}|V|$.
\end{enumerate}\end{thm}

Our approach utilises the theory of crowns in finite groups, which we describe in Section \ref{crowns}. We also require a characterisation of those irreducible linear groups $H\le GL(V)$ such that the set $H^{\ast}(V):=\{h\in H\text{ : }v^h=v\text{ for some }v\in V\backslash\{0\}\}$ is bounded above by an absolute constant, and this is the content of Section \ref{Lin}. Finally, Section \ref{MainProof} is reserved for the proof of Theorem \ref{MainTheorem}.

\section{Crowns in finite groups}\label{crowns}  
Before defining the notion of a crown in a finite group, we require some terminology. First, let $L$ be a monolithic primitive group. That is, $L$ is a finite group with a unique minimal normal subgroup $V\not\le \Frat(L)$. For each positive integer $k$, write $L^k$ for the $k$-fold direct product of $L$. The \emph{crown-based power of $L$ of size $k$} is the subgroup $L_k$ of $L^k$ defined by
$$L_k=\{(l_1, \ldots , l_k) \in L^k  \mid l_1 \equiv \cdots \equiv l_k \ {\mbox{mod}}\text{ } V \}.$$
Equivalently, $L_k=V^k \diag L^k$. 

Next, let $G$ be a finite group. We say that a group $V$ is a \emph{$G$-group} if $G$ acts on $V$ via automorphisms. Following  \cite{paz},  we say that
two irreducible $G$-groups $V_1$ and $V_2$  are  \emph{$G$-equivalent} and we put $V_1 \sim_G V_2$, if there are
isomorphisms $\phi: V_1\rightarrow V_2$ and $\Phi: V_1\rtimes G \rightarrow V_2\rtimes G$ such that the following diagram commutes:

\begin{equation*}
\begin{CD}
1@>>>V_{1}@>>>V_{1}\rtimes G@>>>G@>>>1\\
@. @VV{\phi}V @VV{\Phi}V @|\\
1@>>>V_{2}@>>>V_{2}\rtimes G@>>>G@>>>1.
\end{CD}
\end{equation*}

\

Note that two $G$\nobreakdash-isomorphic
$G$\nobreakdash-groups are $G$\nobreakdash-equivalent. In the abelian case, the converse is true:
if $V_1$ and $V_2$ are abelian and $G$\nobreakdash-equivalent, then $V_1$
and $V_2$ are also $G$\nobreakdash-isomorphic.
It is proved (see for example \cite[Proposition 1.4]{paz}) that two  chief factors $V_1$ and $V_2$ of $G$ are  $G$-equivalent if and only if  either they are  $G$-isomorphic, or there exists a maximal subgroup $M$ of $G$ such that $G/\core_G(M)$ has two minimal normal subgroups $N_1$ and $N_2$
$G$-isomorphic to $V_1$ and $V_2$ respectively. For example, the  minimal normal subgroups of a crown-based power $L_k$ are all $L_k$-equivalent.

Let $V=X/Y$ be a chief factor of $G$. A complement $U$ to $V$ in $G$ is a subgroup $U $ of $G$ such that $UV=G$ and $U \cap X=Y$. We say that   $V=X/Y$ is a Frattini chief factor if  $X/Y$ is contained in the Frattini subgroup of $G/Y$; this is equivalent to saying that $V$ is abelian and there is no complement to $V$ in $G$.
The  number of non-Frattini chief factors $G$-equivalent to $V$ in any chief series of $G$  does not depend on the series, and so this number is well-defined: we will write it as  $\delta_V(G)$. We now define $L_V$, the  monolithic primitive group associated to $V$, by
$$L_{V}:=
\begin{cases}
V\rtimes (G/C_G(V)) & \text{ if $V$ is abelian}, \\
G/C_G(V)& \text{ otherwise}.
\end{cases}
$$
If $V$ is a non-Frattini chief factor of $G$, then $L_V$ is a homomorphic image of $G$.
More precisely,  there exists
a normal subgroup $N$ of $G$ such that $G/N \cong L_V$ and $\soc(G/N)\sim_G V$. Consider now  all the normal subgroups $N$ of $G$ with the property that  $G/N \cong L_V$ and $\soc(G/N)\sim_G V$:
the intersection $R_G(V)$ of all these subgroups has the property that  $G/R_G(V)$ is isomorphic to the crown-based  power $(L_V)_{\delta_V(G)}$.
The socle $I_G(V)/R_G(V)$ of $G/R_G(V)$ is called the $V$-crown of $G$ and it is  a direct product of $\delta_V(G)$ minimal normal subgroups $G$-equivalent to $V$. 

We now record a lemma and two propositions which will be crucial in our proof of Theorem \ref{MainTheorem}. The lemma reads as follows.
\begin{lemma}{\cite[Lemma 1.3.6]{classes}}\label{corona}
	Let $G$ be a finite  group with trivial Frattini subgroup. There exists
	a chief factor $V$ of $G$ and a non trivial normal subgroup $U$ of $G$ such that $I_G(V)=R_G(V)\times U.$
\end{lemma}

To state the propositions, we need some additional notation. For a finite group $G$, and an abelian chief factor $V$ of $G$, set $H_{V}=H_{V}(G):=G/C_G(V)$, $m=m_V=m_{V}(G):=\dim_{\End_{G}(V)}\h(H_{V},V)$, and write $H^{\ast}=H^{\ast}({V})=H^{\ast}_{G}(V)$ for the set of elements $h$ of $H_{V}$ which fix a non-zero vector in $V$. Also, let $\delta_{V}=\delta_V(G)$, and set $\theta_{V}=\theta_V(G)=0$ if $\delta_{V}=1$, and $\theta_{V}=1$ otherwise. Finally, let $q_{V}=q_{V}(G):=|\End_{G}(V)|$ and $n_{V}=n_{V}(G):=\dim_{\End_{G}(V)}{V}$. Note that $\End_{G}(V)$ is a finite field, since $V$ is finite and irreducible.
\begin{prop}{\cite[Proposition 8 and the Proof of Theorem 1]{ALGT}}\label{CheboProp1} Let $G$ be a finite group with trivial Frattini subgroup, and let $U$, $V$ and $R=R_G(V)$ be as in Lemma \ref{corona}. If $U$ is non-abelian, then there exists absolute constants $b_1$, $b_2$ and $b_3$ such that 
$$C(G)\le C(G/U)+\lceil b_{3}(\log{|{G}|})^{2}\rceil +\frac{b_{1}}{b_{2}}\sqrt{|{G}|^{3}}\log{|G|}(1-b_{2}/\log{|{G}|})^{\lceil b_{3}(\log{|{G}|})^{2}\rceil}.$$
\end{prop}

\begin{prop}{\cite[Proposition 8 and the Proof of Theorem 1]{ALGT}}\label{CheboProp} Let $G$ be a finite group with trivial Frattini subgroup, and let $U$, $V$ and $R=R_G(V)$ be as in Lemma \ref{corona}. Suppose that $V$ is abelian, and write $q=q_V$, $n=n_V$ and $H=H_V$, $H^{\ast}=H^{\ast}(V)$ and $m=m_V$. Also, set $\delta=\delta_{V}$ and $\theta=\theta_V$. Set
	
	$$\alpha_U:=\begin{cases}\sum_{0\leq i\leq \delta-1}\frac{q^\delta}{q^\delta-q^i}\leq \delta+\frac{q}{(q-1)^2}& \text { if }H=1,\\
	\min\left\{\left(\delta\cdot \theta+m+\frac{q}{q-1}\right)\frac{|H|}{|H^{\ast}|},\left(\lceil\frac{\delta\cdot \theta}{n}\rceil+\frac{q^n}{q^n-1}\right)|H|\right\}& \text { otherwise.}
	\end{cases}$$
	Then
	$$C(G)\le C(G/U)+\alpha_U.$$
\end{prop}

We conclude this section with the theorem of the first author mentioned in the introduction.
\begin{thm}{\cite[Main Theorem]{AL}}\label{CheboThm} There exists an absolute constant $C$ such that $C(G)\le C\sqrt{|G|}$ for any finite group $G$.\end{thm}

\section{Irreducible linear groups with few elements fixing a non-zero vector}\label{Lin}
Let $V$ be a finite dimensional vector space over an arbitrary field. In this section, our aim is to characterise the groups $H\le GL(V)$, such that the set of elements which fix at least one non-zero vector in $V$ has cardinality bounded above by an absolute constant. For ease of notation, we will write
$$H^{\ast}=H^{\ast}(V):=\{h\in H\text{ : }v^h=v\text{ for some }v\in V\backslash\{0\}\}$$
for such a subgroup $H$. Our main result reads as follows.
\begin{prop}\label{imprim} Let $V$ be a vector space of dimension $n$ over a field $F$, and fix a constant $c>0$. Suppose that $H$ is an irreducible subgroup of $GL(V)$ with the property that $|H^{\ast}|\le c$. Then there exists positive integers $m$ and $k$ such that $n=mk$, and $H\le R\wr \Sym(k)$, where either $|R|$ has order bounded above by a function of $|H^{\ast}|$, or $R\cong \Gamma_1(F_m)$ for some extension field $F_m$ of $F$ of degree $m$.\end{prop}

Proposition \ref{imprim} will follows almost immediately from our next result. Recall that if $F$ is a field, then an irreducible subgroup $H$ of a linear group $GL_n(F)$ is called \emph{weakly quasiprimitive} if every characteristic subgroup of $G$ is homogeneous.
\begin{prop}\label{prim} There exists a function $f:\mathbb{N}\rightarrow \mathbb{N}$ such that if $F$ is a field, $n$ is a positive integer, and $H\le GL_n(F)$ is finite and weakly quasiprimitive, then either $|H|\le f(|H^{\ast}|)$, or $H$ is a subgroup of $\Gamma L_1(F_n)$, for some extension field $F_n$ of $F$ of degree $n$.\end{prop}
\begin{proof} If $n=1$, then $\Gamma L_n(F)=GL_n(F)$. Thus, we may assume that $n>1$. Fix a subgroup $H$ of $GL_n(F)$. We want to prove that if $H$ is not a subgroup of $\Gamma L_1(F_n)$ for some extension field $F_n$ of $F$ of degree $n$, then $|H|$ is bounded in terms of $|H^{\ast}|$.
 
Suppose first that every characteristic abelian subgroup of $H$ is contained in $Z(GL_n(F))$. Let $L$ be the generalised Fitting subgroup of $H$. Our aim is to prove that $|L|$ is bounded above in terms of $|H^{\ast}|$. Since $L$ is self-centralising, this will show that $|H|$ is bounded above in terms of $|H^{\ast}|$, which will give us what we need.

To this end, extend the field $F$ so that $F$ is a splitting field for all subgroups of $L$. Then $L$ may longer be homogeneous, but its irreducible constituents are algebraic conjugates of each other, so $L$ acts faithfully on them. Let $W$ be such a constituent, and let $r_i$, $m_i$, $s_i$, $t_i$, $S_i$ and $T_i$ be as in \cite[Lemma 2.14]{HRD}. By \cite[Lemmas 2.15, 2.16 and 2.17]{HRD}, $W$ decomposes as a tensor product $$W=W_{Z}\otimes W_{r_1}\otimes\hdots\otimes W_{r_a}\otimes W_{s_1}\otimes\hdots\otimes W_{s_b},$$ where $W_Z$ is a $1$-dimensional module for $Z$; $W_{r_i}$ is an irreducible module for $O_{r_i}(G)$ of dimension $r_i^{m_i}$; and $W_{s_i}$ is an irreducible module for $T_i$ of dimension $s_i^{t_i}$. In particular, $[O_{r_i}(H),W_{r_j}]=[T_i,W_{s_j}]=1$ for $i\neq j$, and $[O_{r_i}(H),W_{s_j}]=[T_i,W_{r_j}]=1$, for all $i$, $j$. Hence, if $a+b>1$, then $|L|$ is bounded above in terms of $|H^{\ast}|$, as needed. So we may assume that either $L=Z(G)\circ O_r(H)$, for some prime $r$, or $L=Z(G)\circ T$ is a central product of $t$ copies of a quasisimple group $S$. If $Z(G)\not\le O_r(H)$ in the first case, or $Z(G)\not\le T$ in the second case, then the same argument as above gives that $|L|$ is bounded in terms of $|H^{\ast}|$.

So we may assume that either $L=O_r(H)$, for some prime $r$, or $L=T$ is a central product of $t$ copies of a quasisimple group $S$. Hence, $W$ is a tensor product of 
$m$ [respectively $t$] copies of an irreducible module for an extraspecial group of order $r^3$ [resp. quasisimple group]. Thus, by arguing as in the paragraph above, we can immediately reduce to the case $m=1$ [resp. $t=1$].

Suppose first that $L=O_r(H)=M\rtimes\langle x\rangle$ is extraspecial of order $r^3$, for a prime $r$, where $M$ is cyclic of order $r^2$ if $L$ has exponent $r^2$, and $M$ is elementary abelian of order $r^2$ otherwise. Then, being an absolutely irreducible module for $L$ of dimension $r$, $W$ is isomorphic to $U\uparrow^L_M$, where $U$ is a one dimensional module for $M$ in which $Z(L)$ acts non-trivially. Hence, we may write $W=\bigoplus_{i=0}^{r-1}U\otimes x^i$. It follows that for each non-zero vector $u\in U$, $x^j$ fixes the non-zero vector $u\otimes 1+u\otimes x+\hdots+ u\otimes x^{r-1}$. Thus, $r\le |H^{\ast}|$, from which it follows that $|L|=r^3$ is bounded above in terms of $|H^{\ast}|$, as needed. 

Finally, assume that $L$ is quasisimple. Since $L$ acts on $L^{\ast}$ by conjugation, we may assume that $L^{\ast}\le Z$ (otherwise $L\le \Sym(L^{\ast})$, which would imply that $|L|$ is bounded above in terms of $|H^{\ast}|$). However, since $Z=Z(H)\le Z(GL_n(F))$, $Z$ acts on $V$ by scalar multiplication. Hence, $Z\cap H^{\ast}=1$. It follows that $L^{\ast}=1$, and hence that $L$ is a Frobenius complement in the group $V\rtimes L$. Since $L$ is perfect, it now follows from Zassenhaus' Theorem that $L\cong SL_2(5)$. Whence, $|L|$ is bounded, and this prove sour claim.

Finally, assume that $H$ has a characteristic abelian subgroup not contained in $Z(GL_n(F))$, and let $M\le H$ be maximal with this property. Then by \cite[Lemma 1.10]{LMM}, $M$ is contained in $Z(GL_{\frac{n}{m}}({F}_{m}))$ for some $m$ dividing $n$, and some extension field $F_m$ of $F$ of degree $m$. Hence, $H_1:=C_H(M)$ is a subgroup of $GL_{\frac{n}{m}}(F_m)$ with the property that every characteristic abelian subgroup of $H_1$ is contained in $Z(GL_{\frac{n}{m}}(F_m))$. Furthermore, $H_1$ is weakly quasiprimitive, since it is characteristic in $H$. Also, the group $H/H_1$ is naturally embedded in $\Gal(F_m/F)$, its action induced by a vector space isomorphism $F_m^{\frac{n}{m}}\rightarrow F^n$. Since $H_1^{\ast}(F_m^{\frac{n}{m}})=H_1^{\ast}(F^{n})$, it follows from the arguments above that either $|H_1|$ is bounded in terms of $|H^{\ast}|$; or $n=1$. If $|H_1|$ is bounded in terms of $|H^{\ast}|$, then so is $|H|$, since $H_1$ is self-centralising and normal in $H$. If $n=1$, then $H_1\le GL_1(F_n)$, so $H\le \Gamma L_1(F_n)$, since $H/H_1$ acts on $M=Z(H_1)$ via the Galois group, as described above. This completes the proof.\end{proof}  

Finally, we prove Proposition \ref{imprim}.
\begin{proof}[Proof of Proposition \ref{imprim}] If $H$ is primitive, then the result follows immediately from Proposition \ref{prim}. Thus, we may assume that $H$ is not primitive. Then $V$ may be decomposed into a system $V=W_1\oplus W_2\oplus \hdots \oplus W_k$ of imprimitivity for $H$. Let $\Gamma:=\{W_1,\hdots,W_k\}$, let $S:=H^{\Gamma}$ denote the induced (transitive) action of $H$ on $\Gamma$, and let $R:=\Stab_H(W_1)^{W_1}$ denote the induced action of $\Stab_H(W_1)$ on $W_1$. Then $H$ is isomorphic to a subgroup of the wreath product $R\wr S$.

Finally, since $\Stab_H(W)$ induces $R$ on $W$, we have $|R^{\ast}(W_1)|\le |H^{\ast}(V)|$. Hence, Proposition \ref{prim} implies that either $R\le \Gamma L_1(F_m)$, for some extension $F_m$ of $F$ of degree $m$, or $|R|$ is bounded above by a function of $|H^{\ast}|$. This completes the proof.\end{proof}    

\section{The proof of Theorem \ref{MainTheorem}}\label{MainProof}
We begin our preparations towards the proof of Theorem \ref{MainTheorem} with a lemma concerning the cohomology of an irreducible linear group which has a bounded number of elements fixing a non-zero vector.
\begin{lemma}\label{cohom} There exists an absolute constant $c$ such that if $V$ is a vector space of dimension $n$ over a field $F$ of characteristic $p>0$, and $H$ is an irreducible subgroup of $GL(V)$ with the property that $|H|>\sqrt{|V|}$, then $2^{m}\le c|H^{\ast}|^4$, where $m:=\dim_{F}\Hc^{1}(H,V)$ and $F:=\End_H{V}$.\end{lemma}
\begin{proof} Clearly we may assume that $m>0$. Then, it is proven in \cite[Lemma 9]{AL} that \begin{enumerate}[(1)]
\item $H$ has a unique minimal normal subgroup $N$, which is non-abelian.
\item If $S$ is a component of $H$, then $C_H(S)\subseteq H^{\ast}$.
\item If $W$ is an irreducible $N$-submodule of $V$ not centralised by $S$, then $m\le \dim_{F}\Hc^{1}(S,W)$.\end{enumerate}
Write $N=S_1\times\hdots\times S_t\cong S^t$, and view $H$ as a subgroup in the wreath product $\Aut(N)=\Aut(S)\wr K$, where $K$ denote the induced action of $H$ on the components in $N$. Suppose first that $t>1$. Then (2) implies that $S_i\subseteq H^{\ast}$ for all $i$. Hence, $|H^{\ast}|\geq 1+t(|S|-1)$. Also, $|H^{\ast}|\geq C_H(S_1)\geq |H\cap B||\Stab_K(1)|=|H\cap B|\frac{|K|}{t}$, where $B:=\Aut(S_2)\times\hdots\times\Aut(S_t)$. Note that $|H|\le |H\cap B||\Aut(S)||K|$. It follows that $|H|\le |H^{\ast}|t|\Aut(S)|\le |H^{\ast}|t(|S|-1)^2\le |H^{\ast}|^3$. 

Next, it is shown by Guralnick and Hoffman in \cite[Theorem 1]{GH} that $m\le \frac{n}{2}$. Since we also have $|H|>\sqrt{|V|}$, it follows that 
$$m\le \frac{n}{2}\le \log{\sqrt{|V|}}<\log{|H|}\le\log{|H^{\ast}|^3}.$$
Thus, we may assume that $H\le \Aut(S)$ is almost simple. Before distinguishing cases, we make some remarks. First, $p=\charac{F}$ divides $|H|$, since $\Hc^{1}(H,V)\neq 0$. Furthermore, $|H^{\ast}|\geq |H|_p$, since every element of a Sylow $p$-subgroup of $H$ fixes a non-zero vector in $V$. Finally, note that we may assume that $S$ is not sporadic, since there are a bounded number of such groups having an irreducible module with non-zero cohomology. 

Thus, we have two cases.\begin{enumerate}[(a)]
\item $S\cong \alt(k)$. In this case, we have $\frac{n}{2}\le \log{\sqrt{|V|}}\le \log{|H|}\le k\log{k}$, as long as $k>6$. Hence, by \cite[Proof of Proposition 10]{AL}, we have $m\le 4\log{k}$ and $|H|_p>\frac{k}{2}$, if $k$ is large enough. Hence $2^m\le k^4\le 16|H^{\ast}|^4$ in this case. If $k$ is bounded, then $m$ is also bounded, since $m\le \frac{n}{2}\le \log{|H|}$. Hence, the result also follows in this case.
\item $S\cong ^{\epsilon}X_k(r)$ is a group of Lie type. Write $R_F(S)$ for the smallest degree of a non-trivial irreducible representation of $S$ over the field $F$. If $\charac{F}$ is different to the defining characteristic for $S$, then we have $p^{\frac{R_F(S)}{2}}>|\Aut(S)|$ for $|S|$ large enough (see \cite{LS,SZ,Tiep}). Since $\sqrt{|V|}\le |H|$, we conclude that either $|S|$ is bounded, or $\charac{F}$ coincides with the defining characteristic of $S$. In the latter case, we have $|H|_p>|S|^{\frac{1}{3}}$ by \cite[Proposition 3.5]{KLST}. Also, $|S|\geq |\Aut(S)|^{\frac{4}{5}}$ by \cite[Proposition 4.4]{LPS}. Hence, 
$$|H^{\ast}|>|S|^{\frac{1}{3}}\geq |\Aut(S)|^{\frac{4}{15}}>|H|^{\frac{1}{4}}\geq 2^{\frac{m}{4}}.$$
Thus, either $|S|$ is bounded, or $2^m\le |H^{\ast}|^4$. This gives us what we need.\end{enumerate}\end{proof}

Next, we prove a reduction lemma.
\begin{lemma}\label{Reduction} Fix a constant $\alpha>0$. There exists absolute constants $b=b(\alpha)$, $c=c(\alpha)$ and $c_i=c_i(\alpha)$, $1\le i\le 4$, depending only on $\alpha$, such that: If $G$ is a finite group with trivial Frattini subgroup with the property that $C(G)>\alpha\sqrt{|G|}$, and $U$ is as in Lemma \ref{corona}, then one of the following holds.
\begin{enumerate}[(i)]
\item $U$ is non-abelian and $|G|\le b$.
\item $U$ is abelian and $|U|\le c$.
\item $U$ is abelian and $G$ has a factor group $\ol{G}$ such that\begin{enumerate}[(a)]
\item $\ol{G}\cong V\rtimes H$, with $V\cong U$ an abelian chief factor of $G$, and $H\le GL(V)$;
\item $|H^{\ast}(V)|\le c_1$;
\item $\dim_{\End_H V}{\Hc^{1}(H,V)}\le c_2$; and
\item $c_3|V|\le |H|\le c_4|V|$.
\end{enumerate}\end{enumerate}\end{lemma}
\begin{proof} Adopt in its entirety the notation of Proposition \ref{CheboProp}, so that $U$, $V$ and $R=R_G(V)$ are as in Lemma \ref{corona}. We first consider the case where $V$ is non-abelian. Then by Proposition \ref{CheboProp1} we have 
$$\alpha\sqrt{|G|}< C(G/U)+\lceil b_{3}(\log{|{G}|})^{2}\rceil +\frac{b_{1}}{b_{2}}\sqrt{|{G}|^{3}}\log{|G|}(1-b_{2}/\log{|{G}|})^{\lceil b_{3}(\log{|{G}|})^{2}\rceil},$$
where $b_1$, $b_2$ and $b_3$ are the absolute constants from Proposition \ref{CheboProp1}. Since $C(G/U)\le C\sqrt{|G/U|}$, it follows that $\sqrt{|G|}\le \alpha'\lceil b_{3}(\log{|{G}|})^{2}\rceil +\frac{b_{1}}{b_{2}}\sqrt{|{G}|^{3}}\log{|G|}(1-b_{2}/\log{|{G}|})^{\lceil b_{3}(\log{|{G}|})^{2}\rceil}$, for some constant $\alpha'$ depending only on $\alpha$. Hence, since the square root of $|G|$ divided by the right hand side of the above equation tends to $\infty$ as $|G|$ tends to infinity, we must have that $|G|$ is bounded above by a constant $b=b(\alpha)$ depending only on $\alpha$.

Thus, we may assume that $U$ is abelian. Then by Proposition \ref{CheboProp} and Theorem \ref{CheboThm}, there exists an absolute constant $C$ such that
\begin{align*} \alpha\sqrt{|G|}\le C(G)\le C(G/U)+\alpha_U \le c\sqrt{\frac{|G|}{|U|}}+\alpha_U. \end{align*}
In particular, using the definition of $\alpha_U$ from Proposition \ref{CheboProp}, we conclude that
\begin{align}\label{alpham} \alpha \le \frac{c}{\sqrt{|U|}}+\left(\delta\cdot \theta+m+2\right)\frac{\sqrt{|H|}}{\sqrt{|V|^{\delta}}|H^{\ast}|}\text{, and}\end{align}
\begin{align}\label{alphan} \alpha \le \frac{c}{\sqrt{|U|}}+\left(\left\lceil\frac{\delta\cdot \theta}{n}\right\rceil+2 \right)\frac{\sqrt{|H|}}{\sqrt{|V|^{\delta}}}. \end{align}
We claim first that $\delta=1$. Indeed, assume otherwise, and note that $\frac{|H|}{|H^{\ast}|}\le |H|/|H_v|\le |V|$, for any non-zero $v\in V$. Hence, since $m\le \frac{n}{2}$, we conclude from (\ref{alpham}) that
\begin{align}\label{Con1} |V|^{\frac{\delta-1}{2}}\le C_1(n+\delta), \end{align}
where $C_1=C_1(\alpha)$ depending only on $\alpha$.
Now, since $|U|=|V|^{\delta}=q^{n\delta}$, we conclude that there exists a constant $c=c(\alpha)$ such that if $|U|>c$ and $\delta>1$ then $|V|^{\frac{\delta-1}{2}}> C_1(n+\delta)$. 

Hence, we may assume that $\delta=1$. We will first prove that the properties (b) and (c) of Part (iii) of thge statement of the lemma hold in the factor group $\ol{G}:=G/R_G(V)$. If $|H|\le \frac{|V|}{n^2}$, then (\ref{alpham}) [respectively (\ref{alphan})] implies that $|H^{\ast}|$ [resp. $n$] is bounded above by a constant depending only on $\alpha$. Properties (b) and (c) then follow immediately. 

So we may assume that $|H|> \frac{|V|}{n^2}$. We then use (\ref{alpham}) and the fact that $|H|/|H_v|\le |V|$ to deduce that $|H^{\ast}|\le C_2(1+m^2)$, where $C_2=C_2(\alpha)$ is a constant depending only on $\alpha$. Since $|H|>\sqrt{|V|}$, if follows from Lemma \ref{cohom} that $|H^{\ast}|\le C_3(1+\log{|H^{\ast}|}^2)$, where $C_3=C_3(\alpha)$ is a constant depending only on $\alpha$. It follows that $|H^{\ast}|$, and hence $m$, are bounded above by constants depending only on $\alpha$. This proves that Properties (b) and (c) hold.

Finally, the existence of $c_3$ follows immediately from (\ref{alphan}), while the existence of $c_4$ follows from (\ref{alpham}) and the bound $|H|/|H^{\ast}|\le |V|$. This proves that Property (d) holds, and completes the proof.\end{proof}

We are now ready to prove Theorem \ref{MainTheorem}.
\begin{proof}[Proof of Theorem \ref{MainTheorem}] Let $C$ be the constant from Theorem \ref{CheboThm}; let $f$ be the function from Proposition \ref{prim}; let $b_1$, $b_2$ and $b_3$ be the constants from Proposition \ref{CheboProp1}; and let $b=b(\alpha)$ and $c=c(\alpha)$ be the constants from Lemma \ref{Reduction}. Also, let $c_i$, $1\le i\le 4$, be the functions of $\alpha$ from Lemma \ref{Reduction}. Note that we may assume that $f$, $c_1$, $c_2$ and $c_4$ are increasing functions, while $c_3$ is decreasing. Hence, we may also assume that $g$ satisfies $g(\alpha_1\alpha_2)\geq g(\alpha_1)\alpha_2$, for $g\in\{f,c_1\}$. For ease of notation, we will sometimes write $c_i$ in place of $c_i(\alpha)$.

Set $b_4:=\max\{b,\lceil b_{3}(\log{b})^{2}\rceil +\frac{b_{1}}{b_{2}}\sqrt{b^{3}}\log{b}(1-b_{2}/\log{b})^{\lceil b_{3}(\log{b})^{2}\rceil}\}$; $\alpha':=\max\{\alpha,C\}$; $c_5:=\max\{c,\frac{1}{c_3(\alpha')}f(\lfloor c_1(\alpha')\rfloor)^{\frac{c_1(\alpha')}{c_3(\alpha')}}\lfloor \frac{c_1(\alpha')}{c_3(\alpha')}\rfloor!\}$; and $c_6:=(2+c_2)c_5$. Then define \begin{align*}\delta(\alpha):= &\min\{f(\left\lfloor c_1(\beta)\right\rfloor)\text{ : }0<\beta\le \alpha'\}\text{ and }\\
k(\alpha):= &\frac{c_1(\alpha')}{{c_3(\alpha')}}. \end{align*} Finally, set $\beta:=c_3$ and $\gamma:=c_4$. Note that by construction $k$ is an increasing function of $\alpha$, and that 
\begin{align}\label{Imp}\delta(\beta\sqrt{u})\geq\delta(\beta) \sqrt{u}\geq \delta(\alpha)\sqrt{u},\end{align}
whenever $\beta\le \alpha$.

We will now prove by induction on $|G|$ that $G$ has a factor group $\ol{G}$ such that\begin{enumerate}[(i)]
\item $\ol{G}\cong V\rtimes H$, with $V\cong \mathbb{F}_q^{k}$, and $H\le \Gamma L_1(q)\wr\Sym(k)$, with $q$ a prime power and $k\le k(\alpha)$;
\item $|\ol{G}|\geq \delta(\alpha)\sqrt{|G|}$; and
\item $\beta({\alpha})|V|\le |H|\le \gamma({\alpha})|V|$.
\end{enumerate}

Suppose first that $\Frat(G)=1$, and let $U$, $V$ and $R=R_V(G)$ be as in Lemma \ref{corona}. We would like to reduce to the case where $|G|>b$ if $V$ is non-abelian, and $|U|>c_5$ if $V$ is abelian. We first deal with the non-abelian case. So assume that $V$ is non-abelian and that $|G|\le b$. In this case, we have $$\alpha\sqrt{|G|}< C(G/U)+b_4\le (1+b_4)C(G/U),$$ 
by Proposition \ref{CheboProp1}. In particular, it follows that $C(G/U)> \alpha_1\sqrt{|G/U|}$, where 
$$\alpha_1:=\frac{\alpha\sqrt{|U|}}{1+b_4}.$$
Note that $\gamma(\alpha_1)\le \gamma(\alpha)$, since $\alpha_1\le \alpha$, and $\gamma$ is an increasing function. Similarly, $k(\alpha_1)\le k(\alpha)$ and $\beta(\alpha)\le \beta(\alpha_1)$. Furthermore, $\delta(\alpha_1)\geq \delta(\alpha)\sqrt{|U|}$ by (\ref{Imp}). The inductive hypothesis now implies that $G$, and hence $G/U$, has a factor group $\ol{G}$ with the desired properties.  

Next, assume that $V$ is abelian, and that $|U|\le c$. Then since $\alpha_U\le c_6$, Proposition \ref{CheboProp} yields $C(G/U)> \alpha_2\sqrt{|G/U|}$, where 
$$\alpha_2:=\frac{\alpha\sqrt{|U|}}{1+c_6}.$$
As above, it now follows from the inductive hypothesis and the definitions of $\delta(\alpha)$ and $k(\alpha)$ that $G$ has a factor group $\ol{G}$ with the desired properties.    

Thus, we may assume that $|G|>b$ if $U$ is non-abelian, and $|U|>c_5\geq c$ otherwise. However, Lemma \ref{Reduction} then implies that $U$ must be abelian, and that $G$ has a factor group $\ol{G}$ such that\begin{enumerate}[(a)]
\item $\ol{G}\cong V\rtimes H$, with $V\cong U$ an abelian chief factor of $G$, and $H\le GL(V)$;
\item $|H^{\ast}(V)|\le c_1(\alpha)$;
\item $\dim_{\End_HV}\Hc^{1}(H,V)\le c_2(\alpha)$; and
\item $c_3(\alpha)|V|\le |H|\le c_4(\alpha)|V|$.
\end{enumerate} 
Furthermore, Lemma \ref{imprim} guarantees the existence of positive integers $m$ and $k$, and a transitive permutation group $S$ of degree $k$, such that $n=mk$ and $H\le R\wr S$, with either $|R|\le f(c_1)$, or $R\le \Gamma L_1(p^m)$. Hence, we just need to prove that $k\le k({\alpha})$. Indeed, if this is true then we must have $R\le \Gamma L_1(p^m)$, since otherwise $|V|\le \frac{1}{c_3(\alpha)}|H|\le \frac{1}{c_3(\alpha)}f(c_1(\alpha))^{\frac{c_1(\alpha)}{c_3(\alpha)}}\lfloor \frac{c_1(\alpha)}{c_3(\alpha)}\rfloor!$, contradicting $|U|>c_5$. 

Now, note that (b) and (d) above imply that the number of orbits of $H$ in its action on $V$ is bounded above by $1+\frac{c_1}{c_3}$. Hence, the number of orbits of $X:=GL_m(p)\wr \Sym(k)$ is bounded above by $1+\frac{c_1}{c_3}$. Then since $GL_m(p)$ has $2$ orbits in its action on the natural module $(\mathbb{F}_p)^m$, it follows that the number of orbits of $X$ on $V$ is precisely the number of orbits of $\Sym(k)$ in its action on the $k$-fold cartesian power $\{0,1\}^k$ by permutation of coordinates. This number is precisely $k+1$. Hence, we have $k+1\le 1+\frac{c_1}{c_3}$, and this completes the proof in the case $\Frat(G)=1$.

Finally, assume that $\Frat(G)>1$. Then $C(G/\Frat(G))=C(G)>\beta\sqrt{|G/\Frat(G)|}$, where $\beta:=\alpha\sqrt{|\Frat(G)|}$. Now, since $\alpha\sqrt{|G|}<C(G/\Frat(G))\le C\sqrt{|G/\Frat(G)|}$, we have $|\Frat(G)|\le (\frac{C}{\alpha})^2$. Hence, $\beta\le C$. The result now follows from the inductive hypothesis and the definitions of $\delta(\alpha)$ and $k(\alpha)$.\end{proof}

\end{document}